\documentclass[11pt,a4paper]{article}
\usepackage{amsfonts,latexsym,amsthm,mathtools,authblk,amsmath}

\newtheorem{thm}{Theorem}[section]
\newtheorem{lem}[thm]{Lemma}

\newtheorem{corollary}[thm]{Corollary}
\newtheorem{prop}[thm]{Proposition}
\newcommand{\RR}{\mathbb{R}}

\def \a{\alpha}
\def \b{\beta}
\def \d{\delta}

\def \g{\gamma}
\def \l{\lambda}

\def \ra{\rightarrow}

\title{On oscillating sticky Brownian motion}
\author{Wajdi Touhami}
\affil[]{University of Tunis El Manar, Higher Institute of Medical Technologies\\9 street Dr. Zouhair Essafi 1006 Tunis, Tunisia\\E-mail address: wajdi.touhami@istmt.utm.tn}
\date{}

\begin{document}

\maketitle
\begin{abstract}
Starting with a Brownian motion, we define and study a novel diffusion process by combining stickiness and oscillation properties. The associated stochastic differential equation, resolvent and semigroup are provided. Also the trivariate density of position, local time and occupation time of this diffusion is obtained explicitly. Furthermore, we give a construction of two Brownian motions with drift and scaling whose difference is an oscillating sticky Brownian motion, up to a multiplicative constant.

\end{abstract}
\section{Introduction}
Let $B=(B_t,t\geq 0)$ be a Brownian motion starting from $x\in\RR$, on a filtered probability space $(\Omega, \mathcal{F}, (\mathcal{F}_t)_{t\geq0}, P)$ satisfying the usual conditions. Let $\sigma_+, \sigma_-,\theta$ be positive numbers and let $m$ be the locally finite strictly positive measure on $\mathbb{R}$ given by$$m(dy)=(\frac{1}{\sigma^2_-}1_{\{y<0\}}+\frac{1}{\sigma^2_+}1_{\{y\geq0\}}) dy+\frac{1}{\theta}\delta_0(dy)$$
Define the additive functional $(\alpha_t, t\geq0)$ by $\alpha_t=\int_\mathbb{R} L^y_t(B)\,m(dy)$, where $L^y_t(B)$ stands for the symmetric local time at $y\in\mathbb{R}$ of $B$. Consider the process $X=(X_t, t\geq0)$ obtained from $B$ as follows $$X_t=B_{\alpha^{-1}},\qquad t\geq 0$$
We call $X$ oscillating sticky Brownian motion of parameters $\sigma_+, \sigma_-,\theta$ and we denote OSBM($\sigma_+$, $\sigma_-$, $\theta$). Our definition is just the recipe used by It\^o and McKean in \cite{IM} to construct a diffusion process with speed measure $m$ from Brownian motion. By the occupation time formula, $\alpha$ may be written as \begin{equation}\label{1.1}\a_t=\frac{t}{\sigma_{-}^2}+(\frac{1}{\sigma_{+}^2}-\frac{1}{\sigma_{-}^2})\,\int_0^t1_{\{B_s\geq0\}} ds+\frac{1}{\theta} L_t(B),\qquad t\geq0\end{equation}
The time change $\alpha^{-1}$ is interpreted as scaling time by $\sigma_{+}^2$ or $\sigma_{-}^2$ depending on whether $B$ is positive or negative and as slowing the Brownian motion when it is at zero. In the well-known special case $\sigma_{+}=\sigma_{-}=1$, $X$ is clearly a sticky Brownian motion. This case is investigated by many authors, see for example \cite{MA,H, HW, EP, B, CC, GR}. An other case of interest is obtained by taking $\sigma_+=1$ and tending $\sigma_-$ to infinity, it is the reflected (one sided) sticky Brownian motion. This diffusion was treated in several papers \cite{HM, W, W1, W2, EP, RS, HCA}. On the other hand if $\theta$ tends to infinity then $X$ becomes the oscillating Brownian motion which is defined and studied in \cite{KW}.\\Let $W=(W_t, t\geq0)$ be a standard Brownian motion. The first aim of this paper is to prove that $X$ is the unique weak solution of the stochastic differential equation 
\begin{equation}
\label{1.2}\left\{\begin{array}{ll}dX_t=(\sigma_- 1_{\{X_s<0\}}+\sigma_{+} 1_{\{X_s>0\}})\,dW_s\\
dL_t(X)=\theta 1_{\{X_s=0\}}\,ds\\X_0=x\end{array}\right.\end{equation}
 where $L_t(X)$ is the symmetric local time at zero of $X$. The equation (\ref{1.2}) is a special case of SDEs with discontinuous noise coefficient and a sticky boundary behaviour at zero. Next we show that $X$ is a solution to some well posed martingale problem. Basing on Skorokhod representations and using the proof method in \cite{W}, we compute the resolvent of $X$ and we deduce its semigroup. Another purpose of this paper is to generalise the result in (\cite{H}, p. 31) of sticky coupling, by adapting time change arguments. Precisely we construct for each $(x_1,x_2)\in\mathbb{R}^2$, a pair $(X,X’)$ of diffusion processes satisfying the following stochastic differential problem
\begin{equation}
\label{1.3}\left\{\begin{array}{lll} X\,\mbox{\it{is a Brownian motion with drift $\beta_1$ and scaling $\sigma_-$ starting from $x_1$}}\\
X’\,\mbox{\it{is a Brownian motion with drift $\beta_2$ and scaling $\sigma_-$ starting from $x_2$}}\\
\frac{X-X’}{\sqrt{2}}\,\mbox{\it{is an OSBM $(\sigma_+,\sigma_-,\sqrt{2}\theta)$}}
\end{array}\right.\end{equation}
where $\sigma_+<\sqrt{2}\sigma_-$ and $\beta_1, \beta_2\in\mathbb{R}$ verifying $|\beta_1-\beta_2|<2\theta$. We prove also that the law of $(X,X’)$ is uniquely specified. In a different sense, sticky coupling of two multi-dimensional and non-degenerate It\^o processes was examined in \cite{EZ}. Finding the joint density of the triplet $(X_t, L_t(X), \Gamma_t(X))$ is the final objective of our research, where $\Gamma_t(X)$ is the occupation time of $X$ defined by $\Gamma_t(X)=\int_0^t 1_{\{X_s\geq0\}} ds$. This trivariate density was studied first in \cite{KS} for the Brownian motion, in \cite{AB} for skew Brownian motion and in \cite{WA} for skew sticky Brownian motion. In addition we deduce the occupation time distribution of $X$. Various papers examining diffusions occupation time are available for the reader, see \cite{KW, AB, CC, WA, SS}.
\section{Stochastic differential equation, semigroup and sticky coupling:}
\begin{thm}\label{thm:2.1}
The stochastic differential equation (\ref{1.2}) has a unique weak solution, which is none other than OSBM($\sigma_+, \sigma_-, \theta$).
\end{thm}
\begin{proof} Consider the additive functional $\alpha_t=\int_\mathbb{R} L_t^y(B)\,m(dy)$. By the occupation time formula we have $\alpha_t=\int_0^t (\sigma^{-2}_-1_{\{B_s<0\}}+\sigma^{-2}_+1_{\{B_s\geq0\}})ds+1/\theta L_t(B)$ for each $t\geq0$. Therefore $t\ra \alpha_t$ is continuous and strictly increasing. The fact that $$min(\sigma^{-2}_-, \sigma^{-2}_+)\,t+1/\theta L_t(B)\leq \alpha_t,\qquad t\geq 0$$
yields $\a_\infty=\infty$. Hence the inverse
\begin{equation}\label{2.1}A_t=\alpha^{-1}_t\end{equation}
is finite for all $t\geq0$. Since $\a=(\a_t)$ is adapted to $(\mathcal{F}_t)$ then each $A_t$ is a stopping time with respect to $(\mathcal{F}_t)$, 
consequently $(A_t, t\geq 0)$ define a time change with respect to $(\mathcal{F}_t)$. Consider the time changed process $X$ defined by $X_t=B_{A_t}$ for each $t\geq0$.
A basic property of local times gives $L_t(X)=L_{A_t}(B)$ for each $t>0$. On the other hand, since the Brownian motion $B$ has no sticky point i.e $\int_0^t 1_{\{B_s=0\}}\,ds=0$, then it follows that
$$\int_0^t 1_{\{X_s=0\}}\,ds=\int_0^{A_t} 1_{\{B_s=0\}}\,d\a_s
= \frac{1}{\theta}\int_0^t1_{\{X_s=0\}}\,dL_s(X)=\frac{1}{\theta} L_t(X)$$
where the last equality holds since the mapping $s\ra L_s(X)$ only increase when $X_s=0$.
On the other hand, using (\ref{1.1}) we get  
\begin{eqnarray}\Gamma_t(X)&=&\int_0^t 1_{\{B_{A_s}>0\}}\,d\a_{A_s}+\int_0^t 1_{\{X_s=0\}}\,ds\nonumber\\&=&
\int_0^{A_t}1_{\{B_s>0\}}\,(\frac{ds}{\sigma_{-}^2}+(\frac{1}{\sigma_{+}^2}-\frac{1}{\sigma_{-}^2})1_{\{B_s\geq0\}} ds+\frac{1}{\theta} dL_s(B))+\frac{1}{\theta} L_t(X)\nonumber\\&=&\frac{1}{\sigma^2_+}\int_0^{A_t}1_{\{B_s\geq0\}}\,ds+\frac{1}{\theta} L_t(X)
\nonumber\\&=&\frac{1}{\sigma^2_+}\Gamma_{A_t}(B)+\frac{1}{\theta} L_t(X)\label{*}
\end{eqnarray}
From (\ref{1.1}), (\ref{2.1}) and (\ref{*}) we obtain
\begin{eqnarray}A_t&=&\sigma^2_- t+(\sigma_{+}^2-\sigma^2_-)\Gamma_t(X)-\frac{\sigma_{+}^2}{\theta} L_t(X)\label{**}\\&=&\sigma^2_-\int_0^t(1_{\{X_s<0\}}+1_{\{X_s\geq0\}})\,ds+(\sigma_{+}^2-\sigma^2_-)\int_0^t 1_{\{X_s\geq0\}}\,ds-\sigma_{+}^2\int_0^t 1_{\{X_s=0\}}\,ds\nonumber\\&=&\int_0^t(\sigma^2_-1_{\{X_s<0\}}+\sigma_{+}^2 1_{\{X_s>0\}})\,ds\nonumber 
\end{eqnarray}
Note that $X$ is a continuous local martingale and $<X>_t=A_t$ for all $t\geq0$. Then, according to (\cite{RY}, Proposition V.3.8), there exists a Brownian motion $\{W_t, t\geq0\}$ such that
$$X_t=X_0+\int_0^t (\sigma^2_-1_{\{X_s<0\}}+\sigma_{+}^2 1_{\{X_s>0\}})^{1/2}\,dW_s=x+\int_0^t (\sigma^2_-1_{\{X_s<0\}}+\sigma_+^2 1_{\{X_s>0\}})\,dW_s$$
and this completes the proof of weak existence. Now we show that uniqueness in law holds for the equation (\ref{1.2}). Suppose that $X$ and $W$ solve (\ref{1.2}) and 
put $A_t=\int_0^t (\sigma_{+}^2 1_{\{X_s>0\}}+\sigma^2_-1_{\{X_s<0\}})\,ds$. Noting that $A_\infty=\infty$ almost surely. Indeed, to verify this claim note by the It\^o-Tanaka formula using (\ref{1.1}) that $|X_t|=|x|+M_t+\theta\int_0^\infty 1_{\{X_s=0\}}\,ds$
where $M=(M_t, t\geq0)$ is a continuous martingale with $<M>_t=A_t\uparrow A_\infty$ as $t\ra\infty$. Putting further $A^0_t=\sigma_{+}^2 \int_0^t 1_{\{X_s=0\}}\,ds$ and remarking that $A_t+A^0_t=\int_0^t(\sigma^2_-1_{\{X_s<0\}}+\sigma_{+}^2 1_{\{X_s\geq0\}})\,ds$ for $t\geq0$. One can see that $A^0_t\uparrow\infty$ on $\{A_\infty<\infty\}$ as $t\ra\infty$. But $|X_t|=|x|+M_t+\frac{\theta}{\sigma_{+}^2}A^0_t\ra\infty$ as $t\ra\infty$ almost surely on $\{A_\infty<\infty\}$ contradicting the fact that $A^0_t\uparrow\infty$ on $\{A_\infty<\infty\}$ unless its probability is zero. This shows that $A_\infty=\infty$ almost surely as claimed.
Since $t\ra A_t$ is continuous and increasing, then its right inverse $t\ra \a_t$ defined by \begin{equation}\label{2.4}\a_t=\inf\{s\geq0: A_s>t\}\end{equation} is finite for all $t\geq 0$ and satisfies the same properties itself. 
Taking into account that $<X>_t=A_t$ and using Dambis-Dubins-Schwartz theorem (\cite{RY}, Theorem V.1.6), we affirm that $B=(X_{\a_t}, t\geq 0)$ is a Brownian motion and \begin{equation}\label{2.5}L_t(B)=L_{\a_t}(X)=\theta\int_0^{\a_t} 1_{\{X_s=0\}} ds\end{equation} 
Therefore we get from (\ref{2.4}), (\ref{2.5}) and (\ref{**}) 
\begin{align}\a_t&=\int_0^{\a_t} 1_{\{X_s<0\}} ds+\int_0^{\a_t} 1_{\{X_s>0\}} ds+\int_0^{\a_t} 1_{\{X_s=0\}} ds\nonumber\\&=\frac{1}{\sigma^2_-}\int
_0^{\a_t} (\sigma^2_-1_{\{X_s<0\}}+\sigma_{+}^2 1_{\{X_s>0\}})\,ds+\frac{\sigma^2_--\sigma_{+}^2}{\sigma^2_-}\int_0^{\a_t} 1_{\{X_s>0\}}\,ds+\frac{1}{\theta} L_t(B)\nonumber\\
&=\frac{t}{\sigma_{-}^2}+(\frac{1}{\sigma_{+}^2}-\frac{1}{\sigma_{-}^2})\big[(\sigma_{+}^2-\sigma_{-}^2)\int_0^{\a_t} 1_{\{X_s\geq0\}} ds+\sigma_{-}^2\int_0^{\a_t} 1_{\{X_s\geq0\}}ds-\sigma_{+}^2\int_0^{\a_t} 1_{\{X_s=0\}} ds\big]\nonumber\\&+\frac{1}{\theta} L_t(B)\nonumber\\&=\frac{t}{\sigma_{-}^2}+(\frac{1}{\sigma_{+}^2}-\frac{1}{\sigma_{-}^2})\,\int_0^{\a_t} 1_{\{X_s\geq0\}}\,dA_s+\frac{1}{\theta} L_t(B)\nonumber\\&=\frac{t}{\sigma_{-}^2}+(\frac{1}{\sigma_{+}^2}-\frac{1}{\sigma_{-}^2})\,\Gamma_t(B)+\frac{1}{\theta} L_t(B)\label{3*}
\end{align}
This shows that $t\ra\alpha_t$ is strictly increasing (and continuous) hence $A_t=\alpha^{-1}_t$ is the proper inverse for $t\geq0$. It follows in particular that $\alpha_{A_t}=t$ so that 
\begin{equation}\label{2.6}X_t=X_{\a_{A_t}}=B_{A_t}\end{equation}
From (\ref{3*}) and (\ref{2.6}) we see that $X$ is a well-determined measurable functional of the Brownian motion $B$. This shows that the law of X solving (\ref{2.1}) is uniquely
determined and this completes the proof of law uniqueness.
\end{proof}
\begin{prop}\label{prop:2.2}
There exists a stochastic process $(X_t, t\geq0)$ defined on some probability space starting at $x\in\mathbb{R}$ such that $X$ is a continuous local martingale and so are the processes
\begin{equation}\label{2.7} X_t^2-\int_0^t(\sigma^2_- 1_{\{X_s<0\}}+\sigma_{+}^21_{\{X_s>0\}})\,ds\end{equation}
\begin{equation}\label{2.8}|X_t|-\theta\int_0^\infty 1_{\{X_s=0\}}\,ds\end{equation}
Moreover the law of $X$ is uniquely specified and is equal to the law of OSBM($\sigma_{+},\sigma_-, \theta$) starting at $x$. 
\end{prop}
\begin{proof}
Let $X$ be a weak solution of the equation (\ref{1.2}) starting at $x$. Then $X$ is a time changed Brownian motion, so it is a continuous local martingale. Let $A_t=\int_0^t({\sigma^2}_- 1_{\{X_s<0\}}+\sigma_{+}^2 1_{\{X_s>0\}})\,ds$. Since $<X>_t=A_t$ then the process given by (\ref{2.7}) is a local martingale. Also, by the It\^o-Tanaka formula we have for each $t\geq0$
$$|X_t|-\theta\int_0^\infty 1_{\{X_s=0\}}\,ds=x+\int_0^\infty (\sigma_{-} 1_{\{X_s<0\}}+\sigma_{+}1_{\{X_s>0\}})\,dW_s$$
which leads to the existence. Now, let $X$ be a stochastic process satisfying the conditions of the proposition. 
Define $\a_t=A_t^{-1}$. Clearly $<X>_t=A_t$ and $L_t(X)=\theta\int_0^\infty 1_{\{X_s=0\}} ds$ for each $t\geq 0$. According to the proof of Theorem \ref{thm:2.1} we see that $X$ is an OSBM$(\sigma_{+},\sigma_-, \theta$).
\end{proof}
For $x\in\mathbb{R}$, we denote by $P^x$ the conditional probability $P(\, \cdot\,|X_0=x)$ and by $E^x$ the expectation with respect to $P^x$. For $x,y\in\mathbb{R}$, $p_t(x, y)$ stands for the transition density of Brownian motion. Denote by 
$$H_t^0=\int_0^t 1_{\{X_s=0\}}\,ds,\;H_t^+=\int_0^t 1_{\{X_s>0\}}\,ds,\;H_t^-=\int_0^t 1_{\{X_s<0\}}\,ds$$
and by $\eta_t^0$, $\eta_t^+$ and $\eta_t^-$ their right inverse respectively.
The proof of the following lemma can be easily adapted from the proof of (\cite{CC}, Lemma 1).
\begin{lem}\label{lem:2.3} There exists independent Brownian motions $(W^+_t, t\geq0)$  and $(W^-_t, t\geq0)$ such that 
$$X^+_{\eta_t^+}=\sigma_{+} W^+_t+\frac{\theta}{2} H_{\eta_t^+}^0,\quad X^-_{\eta_t^-}=\sigma_{-} W^-_t+\frac{\theta}{2} H_{\eta_t^-}^0$$
where $$H_{\eta_t^+}^0=\frac{2\sigma_{+}}{\theta}\sup_{s\leq t}(-W^+_s),\quad H_{\eta_t^-}^0=\frac{2\sigma_{-}}{\theta}\sup_{s\leq t}(-W^-_s)$$
\end{lem}
\begin{prop}\label{prop:2.4}
Let $X$ be an OSBM($\sigma_{+},\sigma_-, \theta$) and $r=\frac{1}{2}(\sigma^{-1}_-+\sigma^{-1}_+)$. Then the transition kernel $Q_t(x,dy)$ of $X$ is given by
$$Q_t(x,dy)=\left\{ \begin{array}{ll} {\Huge \displaystyle \sigma^{-2}_{+}g(t,\frac{x+y}{\sigma_+})dy+\theta^{-1} g(t,\frac{x}{\sigma_+})\d_0(dy)+p_0(t,x,y) dy}\quad\it{if}\,x,y\geq0\\\\ 
{\Huge \displaystyle \sigma^{-2}_{-} g(t,\frac{x}{\sigma_{+}}-\frac{y}{\sigma_{-}})dy+\theta^{-1} g(t,\frac{x}{\sigma_+})\d_0(dy)}\quad\it{if}\,x\geq 0,y<0\\\\
{\Huge \displaystyle\sigma^{-2}_{-} g(t,-\frac{x+y}{\sigma_-})dy+\theta^{-1} g(t,-\frac{x}{\sigma_-})\d_0(dy)+p_0(t,x,y) dy}\quad\it{if}\,x,y<0\\\\
{\Huge \displaystyle \sigma^{-2}_+ g(t,-\frac{x}{\sigma_{-}}+\frac{y}{\sigma_{+}})dy+\theta^{-1} g(t,-\frac{x}{\sigma_-})\d_0(dy)}\quad\it{if}\,x,y<0\end{array}\right.$$
\noindent{where} $p_0(t,x,y)=[p(t\sigma^2_{+},x,y)+p(t\sigma^2_{+},x,-y)]1_{\{x,y>0\}}+[p(t\sigma^2_{-},x,y)+p(t\sigma^2_{-},x,-y)]1_{\{x,y<0\}}$\\
and $$g(s,z)=\theta\exp(2 r \theta z+\theta^2 r s)\,\mbox{\it {erfc}}(\frac{z}{\sqrt{2s}}+r \theta \sqrt{2s}),\quad s\geq 0, z\in\mathbb{R}$$
\end{prop}
\begin{proof}
We are guided by (\cite{W}, Proposition 13) and also (\cite{CC}, Lemma 2). Let $W$ be a Brownian motion so that $(X, W)$ is a solution of (\ref{1.2}). Take exponentially distributed random times $T_1,T_2$ and $T_3$ with rate $\l$, independently from $(X, W)$ and define 
\begin{equation}\label{2.9}T=\eta^0_{T_1}\wedge\eta^+_{T_2}\wedge\eta^-_{T_3}\end{equation}
By Lemma \ref{lem:2.3}, $\frac{\theta}{2\sigma_{+}}\, H_{\eta_t^+}^0$ and $\frac{\theta}{2\sigma_{-}}\, H_{\eta_t^-}^0$ are 
the running supremums of independent Brownian motions $W^+$ and $W^-$ respectively. Therefore $H_{\eta_{T_2}^+}^0$ and $H_{\eta_{T_3}^-}^0$ are exponentially distributed with rate $\frac{\theta}{2\sigma_{+}}\gamma$ and $\frac{\theta}{2\sigma_{-}}\gamma$ respectively, where $\g=\sqrt{2\l}$.
Note that $T$ is exponentially distributed with rate $\l$ as for the case given in \cite{W}. Now we have $X_T=0$ if and only if $T=\eta^0_{T_1}$ which can be true only
 if $\eta_{T_1}^0<\eta_{T_2}^+\wedge\eta_{T_3}^-$. Since $H^0_t$ is right continuous and increasing process then the inequality $\eta_{T_1}^0<\eta_{T_2}^+\wedge\eta_{T_3}^-$ implies $T_1<H^0_{\eta^+_{T_2}}\wedge H^0_{\eta^-_{T_3}}$. Thus we get
$$P(X_T=0)=\frac{\l}{\l+\g r \theta}$$ 
For $X_T>0$, we should have $T=\eta_{T_2}^+$ and this can be possible only if $H_{\eta_{T_2}^+}^0<T_1\wedge H_{\eta^-_{T_3}}^0$. Note that ${\sigma^{-1}_{+}} X_{\eta_{T_2}^+}^+$ is a reflected Brownian motion and its probability law for $y>0$ is independent from the level of the running supermum of $-W^+$, which is equal to $H_{\eta_{T_2}^+}^0$ by Lemma \ref{lem:2.3}. Moreover ${\sigma^{-1}_{+}} X_{\eta_{T_2}^+}^+$ is independent of $H_{\eta^-_{T_3}}^0$ 
 due to the Lemma \ref{lem:2.3} again. Hence ${\sigma^{-1}_{+}} X_{\eta_{T_2}^+}^+$ is independent of the event $(H_{\eta_{T_2}^+}^0<T_1\wedge H_{\eta^-_{T_3}}^0)$. Then we have for $y>0$ 
\begin{eqnarray*}
P(X_T\in dy)&=& P(\sigma^{-1}_{+} X_{\eta_{T_2}^+}^+\in dy,\,H_{\eta_{T_2}^+}^0<T_1\wedge H_{\eta^-_{T_3}}^0)\\&=& P(\sigma^{-1}_{+} X_{\eta_{T_2}^+}^+\in dy)\,P(A_{\eta_{T_2}^+}^0<T_1\wedge H_{\eta^-_{T_3}}^0)\\&=&
\frac{\g}{\sigma_{+}} \exp(-\frac{\g }{\sigma_{+}}y)\;\frac{\frac{\theta\g}{2\sigma_{+}}}{\l+\g r \theta}=\frac{\l\theta \exp(-\frac{\g }{\sigma_{+}}y)}{\sigma_{+}^2(\l+\g r \theta)}
\end{eqnarray*}
For $X_T<0$, we should have $T=\eta^-_{T_3}$ and this is can be possible only if $H_{\eta_{T_3}^-}^0<T_1\wedge H_{\eta^+_{T_2}}^0$. Note that ${\sigma^{-1}_{-}} X_{\eta_{T_3}^-}^-$ is a 
reflected Brownian motion and it is independent of the event $(H_{\eta_{T_3}^-}^0<T_1\wedge H_{\eta^+_{T_2}}^0)$ for similar reasons. Analogously we have for $y<0$
$$P(X_T\in dy)=\frac{\l\theta \exp(-\frac{\g }{\sigma_-}y)}{\sigma^2_-(\l+\g r \theta)}$$
We put $\rho=\l+\g r \theta$. Since $R_\l(0,dy)=1/\l\, P(X_T\in dy)$ for each $\l>0$ it follows that
$$R_\l(0,dy)=\frac{\theta}{\sigma^2_-\rho} \exp(-\frac{\gamma }{\sigma_-}y)dy+\frac{1}{\rho}\d_0(dy) 1_{\{y<0\}}+\frac{\theta}{\sigma_{+}^2\rho} \exp(-\frac{\g }{\sigma_{+}}y)dy+\frac{1}{\rho}\d_0(dy) 1_{\{y\geq0\}}$$
Let $T_0=\inf\{s\geq0:X_s=0\}$ be the first time for $X$ to reach 0 and let $R^0_\l(x,dy)$ be the transition kernel of a killed oscillating Brownian motion at 0, which is given by
$$R^0_\l(x,dy)= \left\{ \begin{array}{ll} {\Huge \displaystyle (\g \sigma_+)^{-1}(\exp(-\g \sigma^{-1}_+ |x-y|)-\exp(-\g{\sigma^{-1}_+} (x+|y|))}\,dy,\quad x\geq 0\\\\
{\Huge \displaystyle (\g \sigma^{-1}_-)(\exp(-\g \sigma^{-1}_- |x-y|)-\exp(-\g \sigma^{-1}_- (-x+|y|))}\,dy,\quad x< 0\end{array}\right.$$
By applying the well-known first passage time formula and by taking into consideration that $X$ behaves as an oscillating Brownian motion before reaching 0, we get for all $x\in\RR$
\begin{equation}\label{2.10}R_\l(x,dy)=R^0_\l (x,dy)+E^x(e^{-\l T_0})\,R_\l (0,dy)\end{equation}
Noting that (\cite{KW}, p. 305)\begin{equation}\label{2.11}E^x(e^{-\l T_0})=\exp(-\sqrt{2\l}\,\frac{x}{\sigma_{-}}) 1_{]-\infty,0[}(x)+\exp(-\sqrt{2\l}\,\frac{x}{\sigma_{+}}) 1_{[0,\infty[}(x)\end{equation}
Therefore (\ref{2.10}) together with (\ref{2.11}) yields
$$ R_\l(x,dy)=R^0_\l (x,dy)+\left\{ \begin{array}{ll} {\Huge \displaystyle \frac{\theta}{{\sigma^2_{+}}\rho} \exp{(-\g(\frac{x}{\sigma_+}+\frac{y}{\sigma_+}))}\, dy+\frac{1}{\rho}\exp(-\frac{\g x}{\sigma_+})}\,\d_0(dy);\,x,y\geq 0\\\\
\Huge \displaystyle \frac{\theta}{{\sigma^2_{-}}\rho} \exp{(-\g(\frac{x}{\sigma_+}-\frac{y}{\sigma_-}))}\, dy+\frac{1}{\rho}\exp(-\frac{\g x}{\sigma_+})\,\d_0(dy);\,x\geq 0, y<0\\\\
{\Huge \displaystyle \frac{\theta}{{\sigma^2_{+}}\rho} \exp{(-\g(\frac{y}{\sigma_+}-\frac{x}{\sigma_-}))}\, dy+\frac{1}{\rho}\exp(\frac{\g x}{\sigma_-})}\,\d_0(dy);\,x<0, y\geq0\\\\
{\Huge \displaystyle \frac{\theta}{{\sigma^2_{-}}\rho} \exp{(-\g(\frac{x}{\sigma_-}+\frac{y}{\sigma_-}))}\, dy+\frac{1}{\rho}\exp(\frac{\g x}{\sigma_-})}\,\d_0(dy);\, x,y<0
\end{array}\right.$$
Using tables, we can invert this Laplace transforms to give the transition kernel $Q_t(x,dy)$ for each $x,y\in\mathbb{R}$.
\end{proof}
\begin{thm}\label{thm:2.2} let $\b_1,\b_2\in\mathbb{R}$ and $\theta,\sigma_-, \sigma_+\in\mathbb{R}^*_+$ satisfying $\sigma_+<\sqrt{2}\sigma_-$ and $|\b_1-\b_2|<
2\theta$. Then, for each starting point $(x_1, x_2)\in\mathbb{R}^2$ there exists a
stochastic process $(X, X')=((X_t,X'_t), t\geq0)$ satisfying (\ref{1.3}). Moreover the law of $(X, X')$ is uniquely determined.

\end{thm}

\begin{proof}
Consider the stochastic differential equation given by
\begin{equation}\label{2.12}\left\{\begin{array}{ll}{\Huge \displaystyle Z_t=\sqrt{2} B_t+\frac{\b_1-\b_2}{2\theta} \,L_t(Z)+(\b_1-\b_2)\,(\frac{1}{\sigma^2_{+}}-\frac{1}{\sigma^2_{-}})\,\Gamma_t(Z)+\frac{\b_1-\b_2}{\sigma^2_{-}}\, t}\\
Z_0=x_1-x_2\end{array}\right.
\end{equation}
where $B=(B_t,t\geq0)$ is a standard Brownian motion defined on some filtered probability
space $(\Omega, \mathcal{F}, (\mathcal{F}_t), P)$. We put for all $t>0$$$M_t=\exp(-\frac{\beta_1-\beta_2}{\sqrt{2}\sigma^2_-}B_t-\frac{(\beta_1-\beta_2)^2}{4\sigma^4_-}\,t)$$
It is well known that $(M_t, t\geq 0)$ is a martingale, so there exists a probability measure $\widehat{P}$ on $(\Omega, \mathcal{F})$ such that for all $A\in\mathcal{F}_t$
$$\widehat{P}|_{\mathcal{F}_t}(A)=E(M_t\,1_{A})$$
Then by Girsanov Theorem the process $\widehat{B}=(\widehat{B}_t, t\geq0)$, given by $\widehat{B}_t=B_t+\frac{\beta_1-\beta_2}{\sqrt{2}\sigma^2_-}\,t$ for all $t>0$, is a Brownian motion under $\widehat{P}$. Therefore the SDE (\ref{2.12}) becomes 
\begin{equation}\label{2.13}\left\{\begin{array}{ll}{\Huge \displaystyle Z_t=\sqrt{2} \widehat{B}_t+(\b_1-\b_2)\,(\frac{1}{\sigma^2_{+}}-\frac{1}{\sigma^2_-})\,\Gamma_t(Z)+\frac{\b_1-\b_2}{2\theta} \,L_t(Z)}\\
Z_0=x_1-x_2\end{array}\right.
\end{equation}
The linearity of local times by positive scalar multiplication and the fact that $|\beta_1-\beta_2|<2\theta$, yields the existence of a unique strong solution of (\ref{2.13}), due to (\cite{EM}, Sec. 2.2.1). Therefore, if there exists a solution to (\ref{2.12}), its law is uniquely specified. To show that there exists a solution to
(\ref{2.12}) we can start with a solution to (\ref{2.13}) and apply Girsanov's Theorem in
reverse. To prove the existence part of the Theorem let $Z=(Z_t,t\geq0)$ be a solution to (\ref{2.12}) defined on $(\Omega, \mathcal{F}, P)$. The fact that $Z$ is a semimartingale yields $\int_0^t 1_{\{Z_s=0\}}\,ds=0$, due to the occupation formula.
Consider the additive functional $\alpha=(\alpha_t, t\geq 0)$ defined by $$\a_t=\frac{t}{\sigma_{-}^2}+(\frac{1}{\sigma_{+}^2}-\frac{1}{\sigma_{-}^2})\,\Gamma_t(Z)+\frac{1}{2\theta} L_t(Z)$$
For the same reasons as in the proof of Theorem~\ref{thm:2.1} we see that $t\ra\alpha_t$ admits an inverse $t\ra A_t$ defined by $A_t=\a^{-1}_t$, which is finite for all $t\geq0$. Moreover we have
\begin{eqnarray}A_t&=&\sigma^2_- t+(\sigma_{+}^2-\sigma^2_-)\Gamma_t(Z_A)-\frac{\sigma_{+}^2}{2\theta} L_t(Z_A)\label{***}\\&=&\int_0^t (\sigma_{-}^2 1_{\{Z_{A_s}<0\}}+\sigma_{+}^2 1_{\{Z_{A_s}>0\}})\,ds\label{4*}\end{eqnarray}
From (\ref{***}) we can write $T_t:=2\sigma_{-}^2t-A_t=\int_0^t(\sigma_{-}^2 1_{\{Z_{{A}_t}<0\}}+(2\sigma_{-}^2-\sigma_{+}^2)1_{\{Z_{{A}_t}\geq0\}})\,ds+\frac{\sigma_{+}^2}{2\theta} L_t(Z_A)$.
Then $t\ra T_t$ is continuous, strictly increasing and $T_\infty=\infty$, because $\sigma_{+}<\sqrt{2}\sigma_{-}$. Therefore its inverse $t\ra\tau_t$ given by $\tau_t=T^{-1}_t$ is finite for all $t\geq0$. Define the process $Z'=(Z'_t,t\geq0)$ by \begin{equation}\label{2.14}Z'_t=\sqrt{2} B'_t+(\b_1+\b_2)\tau_t+x_1+x_2\end{equation}
 where $B'=(B'_t, t\geq 0)$ is a Brownian motion independent of $B$ and defined on the same probability space. We now let $X_t=\frac{1}{2}(Z'_{2\sigma_{-}^2t-A_t}+Z_{A_t})$ and $X'_t=\frac{1}{2}(Z'_{2\sigma_{-}^2t-A_t}-Z_{A_t})$,
then \begin{equation}\label{2.15}X_t=\frac{1}{\sqrt{2}} B'_{2\sigma_{-}^2t-A_t}+\frac{1}{\sqrt{2}} B_{A_t}+\beta_1 t+x_1\quad\mbox{\it{and}}\quad X'_t=\frac{1}{\sqrt{2}} B'_{2\sigma_{-}^2t-A_t}-\frac{1}{\sqrt{2}} B_{A_t}+\b_2 t+x_2\end{equation}
Let $(\mathcal{F}^B_t)$ and $(\mathcal{F}^{B'}_t)$ be the filtrations generated by $B$ and $B'$ respectively and then let $\mathcal{G}_t=\mathcal{F}^B_{A_t}\vee\mathcal{F}^{B'}_\infty$ and $\mathcal{H}_t=\mathcal{F}^{B'}_{2{\sigma_{-}^2}t-A_t}\vee\mathcal{F}^B_\infty$. 
By the independence of $B$ and $B'$ we can show that $B_{A_t}$ is $\mathcal{G}_t$-martingale and $B'_{2\sigma_{-}^2t-A_t}$ is 
$\mathcal{H}_t$-martingale due to the theory of continuous times changes.
The relationship (\ref{2.15}) tell us that
$$\mathcal{F}_t^{X,X'}:=\sigma(X_s,X'_s; 0\leq s\leq t)=\sigma(B_{A_s}, B'_{2\sigma_{-}^2s-A_s};0\leq s\leq t)$$ 
Since $\mathcal{F}_t^{X,X'}\subset \mathcal{G}_t$ and $\mathcal{F}_t^{X,X'}\subset \mathcal{H}_t$ then $B_{A_t}$ and $B'_{2\sigma_{-}^2t-A_t}$ are both $\mathcal{F}_t^{X,X'}$-martingales due to the tower property. 
Therefore the processes $(\frac{1}{\sigma_-}(X_t-\b_1 t), t\geq 0)$ and 
$(\frac{1}{\sigma_-}(X'_t-\b_2 t), t\geq0)$ are both $\mathcal{F}_t^{X,X'}$-martingales. Also $<X/\sigma_->_t=<X'/\sigma_->_t=t$, $X_0=x_1$ and $X'_0=x_2$. Hence 
$X$ is a Brownian motion with drift $\b_1$ and scaling $\sigma_-$ started at $x_1$ and $X'$ is a Brownian motion with drift $\b_2$ and scaling $\sigma_-$ started at $x_2$ 
with to a common filtration $(\mathcal{F}_t^{X,X'})$. Next we observe that \begin{equation}\label{2.16}X_t-X'_t= Z_{A_t},\quad t\geq 0\end{equation}
By Tanaka's formula it is easy to show that $L_t(X-X')=L_{A_t}(Z)$. 
Making use of (\ref{***}), (\ref{4*}) and (\ref{2.16}) we obtain \begin{eqnarray} L_t(X-X')&=&\frac{2\theta}{\sigma_{+}^2} (\sigma_{-}^2 t+(\sigma_{+}^2-\sigma_{-}^2) \Gamma_t(X-X')-A_t)\nonumber
\\&=&\frac{2\theta}{\sigma_{+}^2}[\sigma_{-}^2 t-\sigma_{-}^2 \int_0^t 1_{\{X_s\not=X'_s\}}\,ds+(\sigma_{+}^2-\sigma_{-}^2) \int_0^t 1_{\{X_s=X'_s\}}\,ds]\nonumber\\&=&2\theta \int_0^t 1_{\{X_s=X'_s\}}\,ds\label{2.19}
\end{eqnarray}
By reason of $<X-X'>_t=2 A_t$, it follows from (\ref{4*}), (\ref{2.19}) and Proposition~\ref{prop:2.2} that $\frac{X-X'}{\sqrt{2}}$ is an OSBM$(\sigma_+,\sigma_-, \sqrt{2}\,\theta)$. This completes the proof of existence. Now assume that we have any pair of processes $(X,X')$, defined on some
filtered probability space $(\Omega, \mathcal{F}, (\mathcal{F}_t)_{t\geq 0}, P)$, that satisfy the properties of the
Theorem \ref{thm:2.2}. We now define
$$A_t=\int_0^t (\sigma_{-}^2 1_{\{X_s<X'_s\}}+\sigma_{+}^2 1_{\{X_s>X'_s\}})\,ds$$
and we let $\a_t=\inf\{s\geq 0: A_s>t\}$ and $\tau_t=\inf\{ s\geq 0: 2\sigma_{-}^2s-A_s>t\}$. 
It is possible to show that $A_\infty=\infty$. To prove the uniqueness in law, we must show that
the joint laws of $(X_{\a_t}-X'_{\a_t}, t> 0)$ and $(X_{\tau_t}+X'_{\tau_t}, t> 0)$ are equal
to the joint laws of $Z$ and $Z'$, where $Z$ is the solution to stochastic equation
(\ref{2.12}) and $Z'$ is given by (\ref{2.14}), and also that $\a_t=\sigma_{-}^{-2} t+(\sigma_{+}^{-2}-\sigma_{-}^{-2})\Gamma_t(X_\a-X'_\a)+(2\theta)^{-1} L_t(X_\a-X'_\a)$.
Let $W$ and $W'$ are Brownian motions given by 
$$W_t=\frac{1}{\sigma_-}(X_t-\b_1 t-x_1)\quad\mbox{and}\quad W'_t=\frac{1}{\sigma_-}(X'_t-\b_2 t-x_2)$$
Clearly $W-W'$ and $W+W'$ are martingales, $<W-W'>_t=\sigma_{-}^{-2}<X-X'>_t=2\sigma_{-}^{-2} A_t$ and \begin{eqnarray*}
<W+W'>_t&=&\sigma_{-}^{-2}<X+X'>_t\\&=&\sigma_{-}^{-2}(2<X>_t+2<X'>_t-<X-X'>_t)\\&=&\sigma_{-}^{-2}(4\sigma_{-}^2 t-2A_t)=4t-2\sigma_{-}^{-2} A_t\end{eqnarray*} 
Moreover $<W+W',W-W'>=0$. Thus it follows, from Knight's
Theorem (\cite{RY}, Thm. V.1.9), that $(W_{\a_t}-W'_{\a_t}, t\geq 0)$ and $(W_{\tau_t}+W'_{\tau_t}, t\geq 0)$ are independent and each equal in 
distribution to $(\frac{\sqrt{2}}{ \sigma_{-}} B_t, t\geq 0)$, where B is a
standard Brownian motion. Now observe that 
\begin{equation}\label{2.20} X_{\a_t}-X'_{\a_t}=\sigma_-(W_{\a_t}-W'_{\a_t})+(\b_1-\b_2)\a_t+x_1-x_2\end{equation}
and \begin{equation}\label{2.21}L_t(X_{\a_t}-X'_{\a_t})=L_{\a_t}(X-X')=2\theta\int_0^{\a_t} 1_{\{X_s=X'_s\}}\,ds\end{equation}
Arguing as in the proof of Theorem \ref{thm:2.1}, we obtain from (\ref{2.21}) \begin{eqnarray*}\a_t&=&\int_0^{\a_t} 1_{\{X_s>X'_s\}}\,ds+\int_0^{\a_t} 1_{\{X_s<X'_s\}}\,ds+\int_0^{\a_t} 1_{\{X_s=X'_s\}}\,ds\\&=&
\frac{t}{\sigma_{-}^2}+(\frac{1}{\sigma_{+}^2}-\frac{1}{\sigma_{-}^2}) \int_0^t 1_{\{X_{\a_s}-X'_{\a_s} \geq 0\}}\,ds+\int_0^{\a_t} 1_{\{X_s-X'_s=0\}}\,ds\\&=&
\frac{t}{\sigma_{-}^2}+(\frac{1}{\sigma_{+}^2}-\frac{1}{\sigma_{-}^2}) \Gamma_t(X_\a-X'_\a)+\frac{1}{2\theta} L_t(X_\a-X'_\a)
\end{eqnarray*}
The above, together with (\ref{2.20}), tell us that the process $(X_{\a_t}-X'_{\a_t}, t> 0)$
solves the stochastic equation (\ref{2.12}) and hence it is equal
in distribution to $Z$. Then as $X_{\tau_t}+X'_{\tau_t}=\sigma_-(W_{\tau_t}+W'_{\tau_t})+(\b_1+\b_2)\tau_t+x_1+x_2$, we have that the joint distribution of $(X_{\a_t}-X'_{\a_t}, t> 0)$ and $(X_{\tau_t}+X'_{\tau_t}, t> 0)$ is equal
to the joint distribution of $Z$ and $Z'$, from which uniqueness in law follows.
\end{proof}
\section{Trivariate density for oscillating sticky Brownian motion}
Let $X$ be an OSBM$(\sigma_+,\sigma_-,\theta$) started at $x\in\mathbb{R}$. We remind that $X_t=B_{A_t}$, where $B$ is a Brownian motion started at $x$ and $(A_t, t\geq 0)$ is the inverse of $(\alpha_t, t\geq 0)$ which is defined by (\ref{1.1}). According to (\cite{IM}, Chap. 5), $X$ is a strong Markov process. The distribution of $T_0=\inf\{s\geq0:X_s=0\}$ coincides with that 
of an oscillating Brownian motion which is given by (cf. \cite{KW}, Theorem 2) 
$$P(T_0\in ds)=h(s,\frac{x}{\sigma_+})\, 1_{\{x\geq 0\}}+h(s,-\frac{x}{\sigma_-})\, 1_{\{x\leq 0\}}$$ 
where $$h(s,z)=\frac{z}{\sqrt{2\pi}\,s^{3/2}} \, \exp(-\frac{z^2}{2s}),\; z\geq 0, s>0$$ 
Moreover, by the strong Markov property, $h$ satisfies the following convolution property
\begin{equation}\label{4.1}h(\, .\,, x_1)\ast h(\,.\,, x_2)=h(\,.\,, x_1+x_2),\qquad x_1, x_2>0\end{equation}
\begin{lem}\label{lem:3.1} If $X$ started at zero then the triplets $(X_t, L_t(X), \Gamma_t(X))$ and $(-Y_t, L_t(Y), t-\Gamma_t(Y)+\frac{1}{\theta} L_t(Y))$ have the same law for each $t>0$, where $Y$ is an OSBM$(\sigma_-,\sigma_+,\theta)$.
\end{lem}
\begin{proof}According to (\cite{KS}, p. 820) we have \begin{equation}\label{3.2}(B_t, L_t(B), \Gamma_t(B))\overset{law}{=}(-B_t, L_t(B), t-\Gamma_t(B))\end{equation} From (\ref{1.2}), we can show that $Y=-X$ is an $OSBM(\sigma_-,\sigma_+,\theta)$ starting at 0. Therefore, by Theorem \ref{thm:2.1}, there exists a standard Brownian motion $(\widetilde{B}_t, t\geq 0)$ such that $Y=\widetilde{B}_{\widetilde{A}_t}$, where $\widetilde{A}_t=\sigma^2_{+}t+(\sigma^2_{-}-\sigma^2_{+})\Gamma_t(Y)-\frac{\sigma^2_{+}}{\theta}L_t(Y)$. Hence 
\begin{equation}\label{3.3}\widetilde{A}_t=<\widetilde{B}_{\widetilde{A}_t}>=<Y>_t=<X>_t=A_t\end{equation}
By reason of (\ref{3.2}) and (\ref{3.3}), we obtain \begin{equation}\label{3.4}(B_{A_t}, L_{A_t}(B), \Gamma_{A_t}(B))\overset{law}{=}(-\widetilde{B}_{\widetilde{A}_t}, L_{\widetilde{A}_t}(\widetilde{B}), \widetilde{A}_t-\Gamma_{\widetilde{A}_t}(\widetilde{B}))\end{equation}
It follows from (\ref{3.4}), (\ref{*}) and (\ref{**}) that
\begin{equation}\label{3.5}(X_t, L_t(X), \sigma^2_{+} \Gamma_t(X)-\frac{\sigma^2_{+}}{\theta} L_t(X))\overset{law}{=}(-Y_t, L_t(Y), \sigma^2_{+}t-\sigma^2_{+} \Gamma_t(Y))\end{equation}
The result holds immediately by applying the map $(x,y,z)\rightarrow(x,y,\frac{z}{\sigma^2_{+}}+\frac{y}{\theta})$ on each triplet of the equality (\ref{3.5}).
\end{proof}
\begin{thm}\label{thm 3.2} For ${0<\displaystyle\frac{l}{\theta}\leq \tau\leq t}$ the triplet $(X_t, L_t(X), \Gamma_t(X))$ has the following joint density
$$\phi(t,x,y,l,\tau)=\left\{ \begin{array}{ll} {\Huge \displaystyle \sigma^{-2}_- h(t-\tau,\frac{l/2-y}{\sigma_-})\,h(\tau-\frac{l}{\theta},\frac{l}{2\sigma_+}+\frac{x}{\sigma_+})};\,x\geq 0,y<0\\\\
{\Huge \displaystyle \sigma^{-2}_+ h(t-\tau,\frac{l}{2\sigma_-})\, h(\tau-\frac{l}{\theta},\frac{l/2+y+x}{\sigma_+})\,};\,x\geq 0,y>0\\\\ 
{\Huge \displaystyle \sigma^{-2}_- h(t-\tau,\frac{l/2-y}{\sigma_-})\,h(\tau-\frac{l}{\theta},\frac{l}{2\sigma_+}-\frac{x}{\sigma_-})};\,x\leq 0,y<0\\\\
{\Huge \displaystyle \sigma^{-2}_+ h(t-\tau,\frac{l}{2\sigma_-})\, h(\tau-\frac{l}{\theta},\frac{l/2+y}{\sigma_+}-\frac{x}{\sigma_-})\,};\,x\leq 0,y>0
\end{array}\right.$$

\end{thm}

\begin{proof} The proof is adapted from (\cite{KS}, p. 822-824). Let $a,b,c$ and $y$ be positive reel numbers. Define
$$\psi(t,a,b,c)=E^0\big[\int_0^\infty 1_{[y,\infty[}(X_t)\,\exp(-a t-b \Gamma_t-c L_t)\,dt\big]$$
By using (\ref{1.1}), (\ref{*}) and famous Laplace inverse transforms, it follows from (\cite{KS}, Lemma 2.1) that
\begin{align*}
\psi(t,a,b,c)&=E^0[\int_0^\infty 1_{[y,\infty[}(B_{A_t}) \exp(-at-\frac{b}{{\sigma_+}^2}\Gamma_{A_t}(B)-(\frac{b}{\theta}+c) L_{A_t}(B))\,dt\\&=E^0[\int_0^\infty 1_{[y,\infty[}(B_t) \exp(-a\alpha_t-\frac{b}{{\sigma_+}^2}\Gamma_t(B)-(\frac{b}{\theta}+c) L_t(B))\,d\alpha_t\\
&=E^0[\int_0^\infty 1_{[y,\infty[}(B_t) \exp(-\frac{a}{{\sigma_-}^2}t-(\frac{a+b}{{\sigma_+}^2}-\frac{a}{{\sigma_-}^2})\Gamma_t(B)-(c+\frac{a+b}{\theta}) L_t(B))\,\frac{dt}{\sigma_{+}^2}\\&=\frac{2 \exp(-\dfrac{y}{\sigma_+}\sqrt{2(a+b)})}{\sigma_+\,\sqrt{2(a+b)} (c+\dfrac{a+b}{\theta}+\dfrac{\sqrt{2a}}{2\sigma_-}+\dfrac{\sqrt{2(a+b)}}{2\sigma_+})}\\
&=\int_0^\infty\int_0^\infty\int_0^\infty \frac{\exp (-a t-b\tau-c l)\,l}{{4\pi\sigma_+ \sigma_-(\tau-\frac{l}{\theta})^{\frac{1}{2}}(t-\tau)^{\frac{3}{2}}}} \exp(-\frac{(l/2)^2}{2\sigma^2_-(t-\tau)}-\frac{(y+l/2)^2}{2\sigma^2_+(\tau-\frac{l}{\theta})})1_{\{0<\frac{l}{\theta}<\tau<t\}}\\&\quad\quad\quad\quad\quad\quad\quad\quad\quad\quad\quad\quad\quad\quad\quad\quad\quad\quad\quad\quad\quad\quad\quad\quad\quad\quad\quad\quad\quad\quad\quad\quad\quad\quad d\tau\,dl\,dt
\end{align*}
Therefore we get for $y>0$ and $0<\frac{l}{\theta}<\tau<t$ $$P^0(X_t\geq y, L_t\in dl, \Gamma_t\in d\tau)=\frac{l}{{4\pi\sigma^2_+ \sigma_- (\tau-\frac{l}{\theta})^{\frac{1}{2}}(t-\tau)^{\frac{3}{2}}}} \exp(-\frac{(l/2)^2}{2\sigma^2_-(t-\tau)}-\frac{(y+l/2)^2}{2\sigma^2_+(\tau-\frac{l}{\theta})}) dl d\tau$$
This is the distribution of the pair $(L_t, \Gamma_t)$ on the event $[X_t\geq y]$. By differentiating with respect to $y$, it comes that for $0<\frac{l}{\theta}<\tau<t$
\begin{eqnarray*}\phi(t,0,y,l,\tau)&=&\frac{l (y+l/2)}{4\pi\sigma^3_+ \sigma_-(\tau-\frac{l}{\theta})^{\frac{1}{2}}(t-\tau)^{\frac{3}{2}}} \exp(-\frac{(l/2)^2}{2\sigma^2_-(t-\tau)}-\frac{(y+l/2)^2}{2\sigma^2_+(\tau-\frac{l}{\theta})})\\\\
&=&\sigma^{-2}_+ h(t-\tau,\frac{l}{2\sigma_-})\,h(\tau-\frac{l}{\theta},\frac{y+l/2}{\sigma_+})
\end{eqnarray*}
For $y<0$, we can replace $y$ by $-y, \,\tau$ by $t-\tau+\frac{l}{\theta},\,\sigma_+$ by $\sigma_-$ and $\sigma_-$ by $\sigma_+$ in the previous case, due to the Lemma~\ref{lem:3.1}. Hence we get for $0<\frac{l}{\theta}<\tau<t$ 
\begin{equation}\label{3.6}\phi(t,0,y,l,\tau)=\sigma^{-2}_- h(\tau-\frac{l}{\theta},\frac{l}{2\sigma_+})\,h(t-\tau,\frac{l/2-y}{\sigma_-})\end{equation}
Now suppose that $X_0=x\geq 0$. Using the strong Markov property we obtain for $y>0$
\begin{eqnarray*}
P^x(X_t\in dy, L_t\in dl, \Gamma_t\in d\tau)&=& P^x(X_t\in dy, L_t\in dl, \Gamma_t\in d\tau|T_0\leq \tau)\\&=&\int_0^\tau P^x(X_t\in dy, L_t\in dl, \Gamma_t\in d\tau|T_0\leq s)\,P^x(T_0\in ds)\\
&=& \int_0^\tau P^0(X_{t-s}\in dy, L_{t-s}\in dl, \Gamma_{t-s}\in d\tau-s)\,h(s,\frac{x}{\sigma_+})\,ds
\end{eqnarray*}
which implies, together with (\ref{3.6}), that
\begin{eqnarray*}\phi(t,x,y,l,\tau)&=&\sigma^{-2}_+ h(t-\tau,\frac{l}{2\sigma_-}) \int_0^\tau h(\tau-\frac{l}{\theta}-s, \frac{l/2+y}{\sigma_+})\,h(s,\frac{x}{\sigma_+})ds\\
&=&\sigma^{-2}_+ h(t-\tau,\frac{l}{2\sigma_-}) h(\tau-\frac{l}{\theta}, \frac{l/2+y}{\sigma_+}+\frac{x}{\sigma_+})
\end{eqnarray*}
A similar calculations yields for other cases.
\end{proof}
\begin{corollary}\label{corol:3.3} Let $r=\frac{1}{2}(\sigma^{-1}_++\sigma^{-1}_-)$, then for $0<l<\theta t$ we have
$$P^x(X_t\in dy, L_t\in dl)=\left\{ \begin{array}{ll} {\Huge \displaystyle \sigma^{-2}_- h(t-\frac{l}{\theta}, l\,r+ \frac{x}{\sigma_+}-\frac{y}{\sigma_-})\,dy \,dl};\; x\geq 0, y<0\\\\
{\Huge \displaystyle \sigma^{-2}_+ h(t-\frac{l}{\theta}, l\,r+ \frac{x+y}{\sigma_+})\,dy\,dl+p^0_t(x,y)\,dy\,\d_0(dl)}; x\geq 0, y>0\\\\
{\Huge \displaystyle \sigma^{-2}_- h(t-\frac{l}{\theta}, l\,r- \frac{x}{\sigma_-}-\frac{y}{\sigma_-})\,dy \,dl+p^0_t(x,y)\,dy\,\d_0(dl)}; x\leq 0, y<0\\\\
{\Huge \displaystyle\sigma^{-2}_+ h(t-\frac{l}{\theta}, l\,r-\frac{x}{\sigma_-}+\frac{y}{\sigma_+})\,dy\,dl}; x\leq 0, y>0
\end{array}\right.$$
\end{corollary}

\begin{corollary}\label{corol:3.4}
The occupation time of an OSBM$(\sigma_+,\sigma_-,\theta)$, started at zero, has the following distribution for $0<\tau<t$
$$P^0(\Gamma_t\in d\tau)=\frac{\theta^2}{4\pi\sigma_+\sigma_-(t-\tau)^{\frac{3}{2}}}\int_0^\tau \big(\frac{1}{\sigma_+ \sqrt{\tau-l}}+\frac{t-\tau}{\sigma_-(\tau-l)^{\frac{3}{2}}}\big) l\exp(-\frac{(\theta l)^2}{8\sigma^2_+(\tau-l)}-\frac{(\theta l)^2}{8\sigma^2_-(t-\tau)})\,dl$$
\end{corollary}

\begin{corollary}
The local time of an OSBM$(\sigma_+,\sigma_-, \theta)$, started at zero, has the following distribution
$$P^0(L_t\in dl)=\frac{2 r}{\sqrt{2\,\pi (t-\frac{l}{\theta})}} \exp(-\frac{(l r)^2}{2(t-\frac{l}{\theta})}) 1_{\{0<l<\theta t\}}\,dl$$
\end{corollary}

\end{document}